\documentclass[11pt, reqno]{amsart}
\usepackage{amsmath,amsthm,amssymb,mathrsfs}
\usepackage[abbrev,non-sorted-cites]{amsrefs}
\usepackage{color}
\usepackage{verbatim}
\usepackage{datetime}
\usepackage{hyperref}

\theoremstyle{plain}
  \newtheorem{thm}{Theorem}[section]
  \newtheorem{lem}[thm]{Lemma}
  \newtheorem{prop}[thm]{Proposition}
  \newtheorem{cor}[thm]{Corollary} 
	\newtheorem*{thm*a}{Theorem A}
	\newtheorem*{thm*b}{Theorem B}
\theoremstyle{definition}
  \newtheorem{defn}[thm]{Definition}
	
  \newtheorem{rmk}[thm]{Remark}
  \newtheorem*{ack*}{Acknowledgement}
  \newtheorem*{ques*}{Question}
\theoremstyle{plain}

\numberwithin{equation}{section}

\newcommand\ip[2]{\langle{#1},{#2}\rangle}

\newcommand\oh{\frac{1}{2}}
\newcommand\dd{{\mathrm d}}

\newcommand\w{\wedge}
\newcommand\sm{\sigma}
\newcommand\dt{\delta}

\newcommand\vep{\varepsilon}
\newcommand\vph{\varphi}
\newcommand\om{\omega}
\newcommand\ta{\theta}

\newcommand\kp{\kappa}

\newcommand\ld{\lambda}

\newcommand\Sm{\Sigma}
\newcommand\Gm{\Gamma}
\newcommand\Ld{\Lambda}
\newcommand\Ta{\Theta}

\newcommand\CC{\mathcal{C}}
\newcommand\CO{\mathcal{O}}

\newcommand\CA{\mathcal{A}}
\newcommand\CR{\mathcal{R}}

\newcommand\CJ{\mathcal{J}}

\newcommand\FR{\mathfrak{R}}

\newcommand\RSO{\mathrm{SO}}

\newcommand\BR{\mathbb{R}}
\newcommand\BN{\mathbb{N}}
\newcommand\BZ{\mathbb{Z}}

\newcommand\br{\bar}

\newcommand\ot{\otimes}

\newcommand\bv{\mathbf{v}}

\DeclareMathOperator{\tr}{tr}

\DeclareMathOperator{\Hess}{Hess}

\DeclareMathOperator{\re}{Re}
\DeclareMathOperator{\im}{Im}

\begin{document}

\title{Global uniqueness of the minimal sphere\\in the Atiyah--Hitchin manifold}

\author{Chung-Jun Tsai}
\address{Department of Mathematics\\
National Taiwan University\\ Taipei 10617\\ Taiwan}
\email{cjtsai@ntu.edu.tw}

\author{Mu-Tao Wang}
\address{Department of Mathematics\\
Columbia University\\ New York\\ NY 10027\\ USA}
\email{mtwang@math.columbia.edu}

\date{\usdate{\today}}

\thanks{Supported in part by Taiwan MOST grants 106-2115-M-002-005-MY2 and 106-2918-I-002-025 (C.-J.\ Tsai), and NSF grants DMS-1405152 (M.-T.\ Wang).  Part of this work was carried out when the first author was visiting Korea Institute for Advanced Study, and when both authors were visiting the Department of Mathematics of the Chinese University of Hong Kong.
The authors would like to thank J.~Lotay for drawing their attention to the Atiyah--Hitichin manifold and suggesting the connection to the strong stability condition.}



\maketitle


\section{Introduction}

In this note, we study submanifold geometry of the Atiyah--Hitchin manifold, a double cover of the $2$-monopole moduli space, which plays an important role in various settings such as the supersymmetric background of string theory.  When the manifold is naturally identified as the total space of a line bundle over $S^2$, the zero section is a distinguished minimal $2$-sphere of considerable interest. In particular, there has been a conjecture  \cite[Remark on p.262]{ref_MW} about the uniqueness of this minimal $2$-sphere among all closed minimal $2$-surfaces. We show that this minimal $2$-sphere satisfies the ``strong stability condition" proposed in our earlier work \cite{ref_TsaiW2}, and confirm the global uniqueness as a corollary.

\newcommand\ang{{\frac{\psi}{2}}}
\section{The Atiyah--Hitchin manifold}

We start by reviewing the geometry of the Atiyah--Hitchin manifold which is denoted by $M$ throughout this paper. 
The underlying manifold\footnote{The Atiyah--Hitchin manifold in literature often refers to a $\mathbb{Z}/2$ quotient of $M$ as a bundle over $\mathbb{RP}^2$.  The manifold $M$ here is an ALF space of type $D_1$.} $M$ is a degree $-4$ complex line bundle over $S^2$. Utilizing the standard charts on $S^2$,  $z, w:\mathbb{C}\to S^2$ with $z={1}/{w}$, we consider the following co-frame on the unit circle bundle ($e^{i\psi}\in S^1$) over $S^2$:
\begin{align*}
\sm^1 &= \oh\left({\dd\psi} + 2i\frac{z\dd\br{z} - \br{z}\dd z}{1+|z|^2}\right) ~,
&\sm^2 &= \re\left[\frac{2\,e^{i\ang}\,\dd z}{1+|z|^2}\right] ~,
&\sm^3 &= \im\left[\frac{2\,e^{i\ang}\,\dd z}{1+|z|^2}\right] ~.
\end{align*}
Although there is ambiguity in the definitions of $\sm^2$ and $\sm^3$, $(\sm^2)^2$, $(\sm^3)^2$ and $\sm^2\w\sm^3$ are well-defined. In particular,
$(\sm^2)^2+(\sm^3)^2=\frac{4 |\dd z|^2}{(1+|z|^2)^2}$ represents the standard round metric of constant Gauss curvature $1$ on $S^2$. The $1$-forms $\sm^1, \sm^2$ and $\sm^3$ satisfy the relation $\dd\sm^1 = \sm^2\w\sm^3$, and its cyclic permutations.  On the other chart, $(w,\vph) = (1/z,\psi + 4\arg z)$.

The Riemannian metric on $M$ takes the following form
\begin{align}
\dd s^2 &= \dd r^2 + a^2(\sm^1)^2 + b^2(\sm^2)^2 + c^2(\sm^3)^2
\label{metric1} \end{align}
where $a,b,c$ are functions in $r\in [0, \infty)$.  Denoting by prime $(~)'$ the derivative with respect to $r$, these \emph{coefficient functions} $a$, $b$, and $c$ are determined by the following system of ODE's:
\begin{align}
{a'} &= \frac{a^2-(b-c)^2}{2bc} ~,
&{b'} &= \frac{b^2-(c-a)^2}{2ca} ~,
&{c'} &= \frac{c^2-(a-b)^2}{2ab} ~,
\label{ODE1} \end{align} with the initial conditions $a(0)=0$, $b(0) =- m$, and $c(0)=m$ for a positive constant $m$. The manifold is oriented by $\dd r\w\sm^1\w\sm^2\w\sm^3$. The metric is complete and 
 the variable $r$ is the geodesic distance to the zero section ($r=0$) with respect to \eqref{metric1}.   
 
The zero section, $r=0$, is a $2$-sphere denoted by $\Sigma$ and oriented by $\sm^2\w\sm^3$.  The induced metric is round of radius $m$.  $\Sigma$ is the minimal sphere referred in the title of this paper.

Here are some other basic properties of the coefficient functions; see \cite[ch.10 and 11]{ref_AH}.  When $r>0$, $a$ and $c$ are positive; $b$ is \emph{negative}.  Moreover, $a'$, $b'$ and $c'$ are all positive.  The explicit forms of these functions can be found after a change of variable \cite[Theorem 11.18]{ref_AH}.  However, the explicit forms are not needed in this paper. The key to solve for the explicit solution of \eqref{ODE1} is to rewrite the equations as
\begin{align}
(ca+ab)' &= \frac{2}{abc}(ca)(ab) ~,  & (ab+bc)' &= \frac{2}{abc}(ab)(bc) ~,  & (bc+ca)' &= \frac{2}{abc}(bc)(ca) ~.
\label{ODE2} \end{align}
The logarithmic derivative of Jacobi theta functions obey the same equations, up to the factor $2/(abc)$.  Hence, the solution can be constructed from elliptic integrals.

\subsection{The geometry near the zero section $\Sigma$} \label{sec_nearby}

It is useful to write down the series expansions of the coefficient functions at $r = 0$.  With the initial condition $a(0)=0$, $-b(0) = m = c(0)$, one deduces from \eqref{ODE1} that
\begin{align}
a(r) &= 2r - \frac{1}{2m^2}r^3 + \CO(r^4) ~,
&\begin{split}
b(r) &= -m + \oh r - \frac{3}{8m}r^2 + \CO(r^3) ~, \\
c(r) &= m + \oh r + \frac{3}{8m}r^2 + \CO(r^3) ~.
\end{split} \label{series1} \end{align}

Here is an interesting point to make.  The metric arises as the natural metric on the monopole moduli space \cite[ch.2 and 3]{ref_AH}, and is smooth.  At first glance, it seems a little bit strange that the expansions of $b$ and $c$ have both even and odd degree terms.  To see why, let
\begin{align}
q(r) = c(r)-b(r)  \quad\text{and}\quad p(r) = c(r)+b(r) ~.
\label{ds} \end{align}
Note that $q(r)>0$ for any $r\geq0$, $q(0)=2m$ and $p(0)=0$.  When $r>0$, \eqref{ODE2} implies that $(a\,p)'>0$, and thus $p>0$.  The metric \eqref{metric1} can be rewritten as
\begin{align*}
\dd s^2 &= \dd r^2 + \frac{a^2}{4}\left({\dd\psi} + 2i\frac{z\dd\br{z} - \br{z}\dd z}{1+|z|^2}\right)^2 + \frac{q^2+p^2}{4}\frac{4\,|\dd z|^2}{(1+|z|^2)^2} - (2\,q\,p)\re\left[\frac{e^{i\psi}(\dd z)^2}{(1+|z|^2)^2}\right] ~.
\end{align*}
With aforementioned conditions, the smoothness of the metric near $r=0$ is equivalent to that $a(r)/r$, $p(r)/r$ and $q(r)$ are smooth functions in $r^2$.

Equation \eqref{ODE1} in terms of $a, p$, and $q$ are
\begin{align*}
a' &= \frac{2(a^2-q^2)}{p^2-q^2} ~,  & q' &= \frac{2q(p^2-a^2)}{a(p^2-q^2)} ~,  & p' &= 2+\frac{2p(q^2-a^2)}{a(p^2-q^2)} ~.
\end{align*}
From these equations and the initial conditions, one derives that $a$ and $p = c+b$ are odd functions in $r$, while $q = c-b$ is an even function in $r$.

\begin{rmk}
This property of $a,p,q$ may not been seen in some of the radial parameters used in the literature \cite{ref_AH1,ref_GM,ref_AH}.  Those parameters are good to construct the explicit form of the solution.  However, at the zero section, those parameters only respect the $\CC^k$ topology for some $k\in\BN$, but not the smooth one.
\end{rmk}

\subsection{Connections and the ASD Einstein equation}

We briefly recall the convention for connections and curvatures. For a Riemannian manifold with metric $\ip{\,}{\,}$ and Levi-Civita connection $\nabla$, our convention for the Riemann curvature tensor is
\begin{align*}
R(X,Y,Z,W) = \ip{\nabla_Z\nabla_W Y - \nabla_W\nabla_Z Y - \nabla_{[Z,W]} Y}{X} ~.
\end{align*}

Let $\{e_i\}$ be a local orthonormal frame.  Denote the coefficient $1$-forms of the Levi-Civita connection by $\om_i^j$: $\nabla e_i = \om_i^j\ot e_j$.  Since the frame is orthonormal, $\om_i^j = -\om_j^i$.  Throughout this paper, we adopt the Einstein summation convention that repeated indexes are summed.  Denote the dual co-frame by $\{\om^i\}$; the covariant derivative of the co-frame is $\nabla\om^j = -\om_i^j\ot\om^i$.
It follows that
\begin{align*}
\dd\om^j = -\om_i^j\w\om^i ~.
\end{align*}

The curvature form is
\begin{align}
\FR_i^j = \dd\om_i^j - \om_i^k\w\om_k^j ~.
\label{R_curv1} \end{align}
It is equivalent to the Riemann curvature tensor by the following relation:
\begin{align}
\FR_i^j(X,Y) = R(e_j,e_i,X,Y)
\label{R_curv2} \end{align}
for any two tangent vectors $X$ and $Y$.

For the Atiyah--Hitchin manifold $M$ with the Riemannian metric given by \eqref{metric1}, consider the following orthonormal co-frame:
\begin{align}
\om^0 &= -\dd r ~,  &\om^1 &= a\,\sm^1 ~,  &\om^2 &= b\,\sm^2 ~,  &\om^3 &= c\,\sm^3 ~.
\end{align}
Note that $\om^0\w\om^1\w\om^2\w\om^3$ is the positive orientation.  Their exterior derivatives are
\begin{align*}
\dd\om^0 = 0  ~,\quad  \dd\om^1 = -\frac{a'}{a} \om^0\w\om^1 + \frac{a}{bc}\om^2\w\om^3 ~,
\end{align*}
and the equations for $\dd\om^2$ and $\dd\om^3$ are similar.  It follows that
\begin{align} \begin{split}
\om_0^1 &= -\frac{a'}{a}\om^1 ~, \\
\om_2^3 &= -\oh\frac{b^2+c^2-a^2}{abc}\om^1 ~,
\end{split} & \begin{split}
\om_0^2 &= -\frac{b'}{b}\om^2 ~, \\
\om_3^1 &= -\oh\frac{a^2+c^2-b^2}{abc}\om^2 ~,
\end{split} & \begin{split}
\om_0^3 &= -\frac{c'}{c}\om^3 ~, \\
\om_1^2 &= -\oh\frac{a^2+b^2-c^2}{abc}\om^3 ~.
\end{split} \label{conn1} \end{align}

It is known that on a simply-connected $4$-manifold, the hyper-K\"ahler condition is equivalent to
$0 = \FR_0^1 + \FR_2^3 = \FR_0^2 + \FR_3^1 = \FR_0^3 + \FR_1^2$.
In terms of the curvature decomposition in four dimensions, this means that only the anti-self-dual Weyl curvature could be non-zero.  Note that for $(i,j,k) = (1,2,3)$ and its cyclic permutation,
\begin{align*}
\FR_0^i + \FR_j^k = \dd(\om_0^i + \om_j^k) + (\om_0^j + \om_k^i)\w(\om_0^k + \om_i^j) ~,
\end{align*}
and thus vanishes if
\begin{align}
\om_0^i + \om_j^k = -\sm^i ~.
\label{asd} \end{align}
From \eqref{conn1}, this condition is exactly the equation \eqref{ODE1}.  One can compare with the case of the Eguchi--Hanson metric, where $\om_0^i + \om_j^k$ vanishes.  See, for example, \cite[Section 2]{ref_TsaiW}.

\subsection{Hyper-K\"ahler structure}
Recall that the hyper-K\"ahler structure is characterized by the existence of three linearly independent parallel self-dual $2$-forms. 
 With the orientation $\om^0\w\om^1\w\om^2\w\om^3$, the space of self-dual 2-forms $\Ld^2_+$ is spanned by
$\om^0\w\om^1 + \om^2\w\om^3$, $\om^0\w\om^2 + \om^3\w\om^1$, and $\om^0\w\om^3 + \om^1\w\om^2$. 
From \eqref{asd}, the Levi-Civita connection on $\Ld^2_+$ reads:
\begin{align}
\nabla(\om^0\w\om^1 + \om^2\w\om^3) &= - \sm^3\ot(\om^0\w\om^2 + \om^3\w\om^1) + \sm^2\ot(\om^0\w\om^3 + \om^1\w\om^2) ~, \label{asd1} \\
\nabla(\om^0\w\om^2 + \om^3\w\om^1) &= \sm^3\ot(\om^0\w\om^1 + \om^2\w\om^3) - \sm^1\ot(\om^0\w\om^3 + \om^1\w\om^2) ~, \notag \\
\nabla(\om^0\w\om^3 + \om^1\w\om^2) &= -\sm^2\ot(\om^0\w\om^1 + \om^2\w\om^3) + \sm^1\ot(\om^0\w\om^2 + \om^3\w\om^1) ~. \notag
\end{align}
We proceed to find three linearly independent parallel self-dual 2-forms. 
Consider the following parametrization of $\RSO(3)$:
\begin{align*}
S = \frac{1}{1+|z|^2}\begin{bmatrix}
2\re(z) & \im(e^{-i\ang}+e^{i\ang}z^2) & \re(e^{-i\ang}-e^{i\ang}z^2) \\
2\im(z) & -\re(e^{-i\ang}+e^{i\ang}z^2) & \im(e^{-i\ang}-e^{i\ang}z^2) \\
1-|z|^2 & 2\im(e^{i\ang}z) & -2\re(e^{i\ang}z)
\end{bmatrix} ~.
\end{align*}
The Maurer--Cartan form is
\begin{align*}
S^{-1}\dd S = \begin{bmatrix} 0 & \sm^3 & -\sm^2 \\ -\sm^3 & 0 & \sm^1 \\ \sm^2 & -\sm^1 & 0 \end{bmatrix} ~,
\end{align*}
which is exactly the connection $1$-form in terms of the basis $\{\om^0\w\om^1 + \om^2\w\om^3,\om^0\w\om^2+\om^3\w\om^1,\om^0\w\om^3+\om^1\w\om^2\}$.  

Three parallel self-dual $2$-forms can be obtained by pairing the row vectors of $S$ with the above basis.  It is easier to use the following expressions:
\begin{align}
\om^0\w\om^1 + \om^2\w\om^3 &= -a\,\dd r\w\sm^1 + \frac{p^2-q^2}{4}\frac{2i\,\dd z\w\dd\br{z}}{(1+|z|^2)^2} ~, \label{cal} \\
(\om^0\w\om^2 + \om^3\w\om^1) + i(\om^0\w\om^3 + \om^1\w\om^2) &= \frac{(p\,e^{i\ang}\,\dd z - q\,e^{-i\ang}\,\dd\br{z})\w(\dd r - ia\,\sm^1)}{1+|z|^2} \notag
\end{align}
where $p$ and $q$ are defined by \eqref{ds}.  Then, the $[\text{3rd row}]$ of $S$ gives
\begin{align} \begin{split}
& \frac{1-|z|^2}{1+|z|^2}\left[ \frac{(p^2-q^2)}{4}\frac{2i\,\dd z\w\dd\br{z}}{(1+|z|^2)^2} - a\,\dd r\w\sm^1 \right] \\
&\quad - 2\im\left[ \frac{\br{z}\,\dd z\w\left(p\,(\dd r - ia\,\sm^1)\right) - q\,\br{z}\,\dd\br{z}\w\left(e^{-i\psi}\,(\dd r - ia\,\sm^1)\right)}{(1+|z|^2)^2} \right] ~,
\end{split} \label{kform} \end{align}
and $[\text{1st row}] + i\,[\text{2nd row}]$ gives
\begin{align} \begin{split}
& \frac{2z}{1+|z|^2}\left[ \frac{(p^2-q^2)}{4}\frac{2i\,\dd z\w\dd\br{z}}{(1+|z|^2)^2} - a\,\dd r\w\sm^1 \right] \\
&\quad - i\frac{\dd z\w\left(p\,(\dd r - ia\,\sm^1)\right) - q\,\dd\br{z}\w\left(e^{-i\psi}\,(\dd r - ia\,\sm^1)\right)}{(1+|z|^2)^2} \\
&\qquad + i\frac{q\,z^2\dd z\w\left(e^{i\psi}\,(\dd r + ia\,\sm^1)\right) - z^2\,\dd\br{z}\w\left(p\,(\dd r + ia\,\sm^1)\right)}{(1+|z|^2)^2} ~.
\end{split} \label{hvform} \end{align}
Recall that $a(r) = 2r + r^{\text{odd}}$, $p(r) = r + r^{\text{odd}}$ and $q(r) = 2m + r^{\text{even}}$ near $r=0$.  It follows that the $2$-forms \eqref{kform} and \eqref{hvform} are indeed smooth.

From \eqref{kform} and \eqref{hvform}, one sees that the restrictions of the 2-forms to the zero section $\Sigma$ become 
\begin{align*}
\frac{1-|z|^2}{1+|z|^2}\left[ \frac{-2im^2\,\dd z\w\dd\br{z}}{(1+|z|^2)^2}\right]
\quad\text{and}\quad \frac{2z}{1+|z|^2}\left[ \frac{-2im^2\,\dd z\w\dd\br{z}}{(1+|z|^2)^2}\right] ~,
\end{align*}
and thus $\Sigma$ is the ``twistor" sphere.  Namely, it is the parameter space of the K\"ahler forms.  The restriction of any K\"ahler form on $\Sigma$ has zero total integral.  The homology class $[\Sigma]$ is a Lagrangian class with respect to any K\"ahler form.

Denote the complex structure corresponding to the self-dual $2$-form given by the $[i\text{-th row}]$ of $S$ by $J_i$, and the complex structure on $\Sm$ by $J_{S^2}$.  By regarding the embedding of $\Sm$ as a map $u:S^2\to M$, the above computation shows that $J_i\circ\dd u = -x_i\,\dd u \circ J_{S^2}$, where $x_1, x_2$, and $x_3$ are the standard coordinate functions on $S^2$ satisfying $x_1+ix_2 = 2z/(1+|z|^2)$ and $x_3 = (1-|z|^2)/(1+|z|^2)$.  In particular, the map $u$ obeys
\begin{align}
\dd u\circ J_{S^2} &= - x_1\,J_1\circ\dd u - x_2\,J_2\circ\dd u - x_3\,J_3\circ\dd u ~.
\label{qtn} \end{align}

\subsection{Curvatures}
We compute the curvature components of $M$ in this section. Recalling the formula of the $\FR_0^1$ component
\begin{align*}
\FR_0^1= \dd\om_0^1-\om_0^2\wedge \om_2^1-\om_0^3\wedge\om_3^1
\end{align*}
and substituting the connection forms from \eqref{conn1}, we 
derive \[\FR_0^1=\frac{a''}{a}\, \om^0\w\om^1 - \kp(a,b,c) \, \om^2\w\om^3 ~,\] where $\kp(a, b, c)$ is defined by \begin{align}
\kp(a,b,c)\equiv \frac{1}{2(abc)^2} \left[ 2a^4 - a^2(b-c)^2 - a^3(b+c) + a(b-c)^2(b+c) - (b+c)^2(b-c)^2 \right] ~. \label{curvf}
\end{align}

On the other hand, from \eqref{ODE1}, it can be checked that ${a''}/{a}=\kp(a,b,c)$, or $R_{1001} = R_{2301}$, a fact that can be derived 
alternatively from the hyper-K\"ahler condition.  One verifies directly that $\kp(a,b,c) = \kp(a,c,b)$ and $\kp(a,b,c) + \kp(c,a,b) + \kp(b,c,a) = 0$.
Due to the formal cyclic symmetry of $(a,b,c)$, all the non-trivial components of the Riemann curvature tensor are listed as follows (up to the symmetry of the curvature tensor).
\begin{align} \begin{cases}
R_{1001} = R_{2301} = R_{2332} = \kp(a,b,c) = \displaystyle\frac{a''}{a} ~, & \medskip\\
R_{2002} = R_{3102} = R_{3113} = \kp(b,c,a) = \displaystyle\frac{b''}{b} ~, & \medskip \\
R_{3003} = R_{1203} = R_{1221} = \kp(c,a,b) = \displaystyle\frac{c''}{c} ~. &
\end{cases} \label{curv} \end{align}

\subsection{Totally geodesic surfaces} \label{sec_surface}

In \cite[ch.7 and 12]{ref_AH}, two kinds of totally geodesic surfaces are introduced to study the geodesics of the ambient space \cite[ch.13]{ref_AH}.  
\begin{enumerate}
\item In the formulation here, the first kind is the fiber of the $-4$-bundle.  For example, set $z = 0$.  The induced metric is $\dd r^2 + \frac{a^2}{4}\dd\psi^2$.

\item The second kind is topologically a cylinder.  For instance, consider $(r\,e^{i\psi},z) = (s\,e^{-2i\ta}, e^{i\ta})$ for $(s,e^{i\ta})\in\BR\times S^1$.  The induced metric is $\dd s^2 + c^2\dd\ta^2$ for $s>0$, and $\dd s^2 + b^2\dd\ta^2$ for $s<0$.  One may also take the $S^1$-factor to be the great circle, $\{\im z = 0\}$ or $\{\re z = 0\}$, and take the $\BR^1$-factor to be a line on the $re^{i\psi}$-plane with suitable direction.
\end{enumerate}
Each of the above examples is holomorphic with respect to some complex structure.  The readers are directed to \cite{ref_AH} for more discussions.

\section{Geometric properties of the minimal sphere}

\subsection{Strong stability}

The Jacobi operator of the volume functional on a minimal submanifold is $\CJ = (\nabla^\perp)^*\nabla^\perp + \CR -\CA$.  The concrete form of the zeroth order part is
\begin{align*}
(\CR-\CA)(V) &= \sum_{\mu,\nu}\left[ -\sum_{\ell} R_{\ell\mu\ell\nu} V^\mu - \sum_{\ell,k}h_{\mu\ell k}h_{\nu\ell k}V^\mu \right]e_\nu
\end{align*}
on a normal vector $V = \sum_\mu V^\mu e_\mu$.  Here, $k,\ell$ are indices for the orthonormal frame of the tangential part, and $\mu,\nu$ are for the normal part.  In \cite[Definition 3.1]{ref_TsaiW2}, a minimal submanifold is said to be strongly stable if $\CR - \CA$ is \emph{pointwise} positive definite.  It is clear that strong stability implies strict stability, i.e.\ $\CJ$ is a positive operator.  In \cite[Proposition 5.5]{ref_MW}, the minimal sphere $\Sm$ is shown to be strictly stable.  We show that it is indeed strongly stable.

\begin{prop} \label{prop_sstable}
The minimal sphere $\Sm$ in the Atiyah--Hitchin manifold is strongly stable.
\end{prop}

\begin{proof}[Proof 1: direct computation]
Note that the indices $2,3$ are tangential directions, and $0,1$ are normal directions.  According to \eqref{series1} and \eqref{conn1}, the components of its second fundamental form are
\begin{align*}
\frac{1}{2m} = -h_{022} = h_{033} = h_{123} = h_{132}
\qquad\text{and}\qquad  0 = h_{023} = h_{032} = h_{122} = h_{133} ~.
\end{align*}
In \cite[Remark on p.37]{ref_AH}, Atiyah and Hitchin showed that $\Sm$ is not a totally geodesic by representation theory.  By plugging \eqref{series1} into \eqref{curvf},
\begin{align}
\kp(a,b,c) = -\frac{3}{2m^2}  \qquad\text{and}\qquad
\kp(b,c,a) = \kp(c,a,b) = \frac{3}{4m^2}  \quad\text{ at } r=0 ~.
\label{curv0} \end{align}
With \eqref{curv}, the components of $\CR-\CA$ are as follows.
\begin{align*}
-\sum_{j=2}^3 R_{j0j0} - \sum_{j,k=2}^3 h_{0jk}h_{0jk} &= R_{2002} + R_{3003} - (h_{022})^2 - (h_{033})^2 = \frac{1}{m^2} ~, \\
-\sum_{j=2}^3 R_{j1j1} - \sum_{j,k=2}^3 h_{1jk}h_{1jk} &= R_{2112} + R_{3113} - (h_{123})^2 - (h_{132})^2 = \frac{1}{m^2} ~,
\end{align*}
and the off-diagonal part vanishes.  Clearly, $\CR-\CA$ is positive definite.
\end{proof}

There is a calculation-free argument.  Here is the brief explanation.

\begin{proof}[Proof 2: special Lagrangian type argument]
Although the minimal sphere can never be (special) Lagrangian, the argument in \cite[Appendix A.1]{ref_TsaiW2} works as well.  Note that for any $p\in\Sm$, $T_p\Sm$ is a special Lagrangian plane with respect to \emph{some} Calabi--Yau structure.  For instance, when $|z|=1$, $T_p\Sm$ is Lagrangian with respect to \eqref{kform}.  Its phase with respect to \eqref{hvform} is basically $\arg z$.  The computation in \cite[Appendix A.1]{ref_TsaiW2} is tensorial.  By using the complex structure determined by the holomorphic volume form \eqref{hvform}, the computation works at \emph{any} point with $|z|=1$. Since $\Sm$ is the twistor sphere, the argument works everywhere on $\Sm$.  It follows that $\CR-\CA$, as a linear map on the normal bundle,  is 
a multiple of the identity map.
\end{proof}

By applying \cite[Theorem 6.2]{ref_TsaiW2}, the minimal sphere $\Sm$ is $\CC^1$ stable under the mean curvature flow.  
\begin{cor}
There exists an $\vep>0$ which has the following significance.  For any surface $\Gm$ satisfying
$\sup_{q\in\Gm} \left( r^2 (q) + (1+(\om^2\w\om^3)(T_q\Gm)) \right) < \vep $,
the mean curvature flow $\Gm_t$ with $\Gm_0 = \Gm$ exists for all time, and converges smoothly to $\Sm$ as $t\to\infty$.
\end{cor}

Here $r$ is considered to be the distance function to the zero section and the $2$-form $-\om^2\w\om^3$ is parallel along geodesics normal to $\Sm$ by 
 \eqref{conn1}.

\subsection{Estimates on the derivatives}

In order to say some global property of the minimal sphere, a better understanding on the coefficient functions is needed.

\begin{lem} \label{estimate}
The coefficient functions  $a$, $b$, and $c$ of the Atiyah--Hitchin metric \eqref{metric1} obey the following relation.
\begin{align*}
1 > \frac{r\,a'(r)}{a(r)} > \frac{r\,c'(r)}{c(r)} > \frac{-r\,b'(r)}{b(r)} > 0
\end{align*}
for any $r>0$.
\end{lem}

\begin{proof}
This lemma can be proved easily by using the theory established in \cite[ch.9 and 10]{ref_AH}.  The variable $\xi$ in \cite{ref_AH} is the geodesic distance $r$ here.  The key ingredients are summarized as follows.  Atiyah and Hitchin introduced the functions
\begin{align*}
x = \frac{a}{c}  \qquad\text{and}\qquad  y = \frac{b}{c} ~.
\end{align*}
Both $x$ and $y$ can serve as the radial coordinate.  In fact, they mainly use $x$ as the variable in \cite[ch.10]{ref_AH}.  At $r=0$, $(x(0),y(0)) = (0,-1)$, and $(x(r),y(r))\to(1,0)$ as $r\to\infty$.  That is to say, the domain of $x$ is $[0,1)$; the domain of $y$ is $[-1,0)$.  When $r>0$, the curve $(x(r),y(r))$ lies entirely in the region
\begin{align}
y < -1+x ~,\quad  0<x<1 ~,\quad  -1<y<0 ~.
\label{key} \end{align}
The bound $y \leq -1+x$ is given by \cite[Lemma 10.1]{ref_AH}.  From its proof, it is not hard to see that the equality only happens at $(x,y)=(0,-1)$, or $r=0$.  It is also illustrative to give their expansions \eqref{series1} near $r = 0$,
\begin{align*}
x(r) = \frac{2}{m}r - \frac{1}{m^2}r^2 + \CO(r^3)  \qquad\text{and}\qquad
y(r) = -1 + \frac{1}{m}r - \frac{1}{2m^2}r^2 + \CO(r^3) ~.
\end{align*}

The equations \eqref{ODE1} become
\begin{align*}
a' &= \frac{x^2-(y-1)^2}{2y} ~,
& b' &= \frac{y^2-(x-1)^2}{2x} ~,
& c' &= \frac{1-(x-y)^2}{2xy} ~.
\end{align*}
The derivatives of $x(r)$ and $y(r)$ are
\begin{align*}
x' = - \frac{1}{c}\,\frac{(1-x)(1+x-y)}{y}  \qquad\text{and}\qquad
y' = - \frac{1}{c}\,\frac{(1-y)(1+y-x)}{x} ~.
\end{align*}

It follows from \eqref{key} that $b'>0$ when $r>0$.  We compute
\begin{align*}
\frac{c'}{c} + \frac{b'}{b} &= \frac{1}{c}\,\frac{1-x+y}{y} ~, \\
\frac{a'}{a} - \frac{c'}{c} &= \frac{x'}{x} = \frac{1}{c}\,\frac{(1-x)(1+x-y)}{x(-y)} ~.
\end{align*}
According to \eqref{key}, both quantities are positive when $r>0$.

It remains to show that $a\geq r\,a'$.  With \eqref{series1}, $\frac{a}{a'} = r + \frac{1}{2m^2}r^3 + \CO(r^4)$ near $r = 0$.  Hence, $\frac{a}{a'} > r$ for sufficiently small $r$.  The derivative of $\frac{a}{a'} - r$ in $r$ is $\frac{a}{(a')^2}(-a'')$.  By invoking \cite[Lemma 10.10]{ref_AH}, $a''<0$ when $r>0$.  We will say something about their proof momentarily.

To sum up, $\frac{a}{a'}-r$ is monotone increasing in $r$, and is positive for small $r$.  Therefore, it must be positive for any $r>0$.  This finishes the proof of this lemma.
\end{proof}

It follows from \eqref{curv} that
\begin{align*}
a'' &= a\,\kp(a,b,c) \\
&= \frac{1}{c}\,\frac{2x^4 - x^2(y-1)^2 - x^3(1+y) + x(1-y)^2(1+y) - (1-y)^2(1+y)^2}{2xy^2}
\end{align*}
where $\kp$ is defined by \eqref{curvf}.  One can study the maximum of the numerator over the closure of \eqref{key}.  It turns out that the maximum is $0$, and is achieved only at $(0,-1)$ and $(1,0)$.  The argument of \cite[Lemma 10.10]{ref_AH} is cleverer.  They work with
\begin{align*}
a'' &= \left( \frac{x}{y} + \frac{1-x^2-y^2}{2y^2}\,\frac{\dd y}{\dd x} \right) \frac{\dd x}{\dd r} ~,
\end{align*}
and analyze it according to whether $\frac{\dd y}{\dd x}\leq 1$ or not.  The sign of $b''$ is examined in \cite[Lemma 10.19]{ref_AH}; it is negative when $r>0$.  For $c''$, it is positive for small $r$, and negative for large $r$.  See \cite[last paragraph on p.99]{ref_AH}.  Note that the notion of convexity/concavity in \cite{ref_AH} is different from the usual one.  These convexity/concavity properties are directly related to the geometry of the surfaces mentioned in section \ref{sec_surface}.

\subsection{Calibration}

We show that the minimal sphere is actually a minimizer of the area functional.  According to J.~Lotay, this was known to M.~Micallef.  The theory of calibration can be found in \cite[\S II.4]{ref_HL}.

\begin{prop}
The minimal sphere $\Sm$ in the Atiyah--Hitchin manifold is a calibrated submanifold.  Therefore, it minimizes the area within its homology class.
\end{prop}

\begin{proof}
The only task is to construct a closed $2$-form of comass one, whose restriction on $\Sm$ coincides with its area form.  Take $\Ta = m^2\,\sm^2\w\sm^3 = \frac{-m^2}{bc}\om^2\w\om^3$.  From the expression $m^2\,\sm^2\w\sm^3$, it is easy to see that $\dd\Ta = 0$ and $\Ta|_\Sm = {\rm dvol}_\Sm$.

It remains to check that comass one condition.  According to Lemma \ref{estimate}, $(bc)'<0$ when $r>0$.  It follows that $bc \leq -m^2$ for any $r$, which implies that $\Ta$ has comass one.
\end{proof}

\subsection{Two-convexity of the distance function}

In this section, we apply the barrier function argument to prove the rigidity of the minimal sphere in the Atiyah--Hitchin manifold.  Here is a simple fact in linear algebra.

\begin{lem} \label{linear}
Let $Q$ be a symmetric matrix on $\BR^n$, with eigenvalues $\ld_n\geq\cdots\geq\ld_2\geq\ld_1$.  Fix $k\in\{1,\cdots,n\}$.  Then, the minimum of
\begin{align*}
\left\{\, \tr_L(Q) ~\big|~ L\subset\BR^n \text{ is a vector subspace of dimension } k \,\right\}
\end{align*}
is exactly $\sum_{j=1}^k\ld_j$.
\end{lem}

\begin{proof}
Regard the domain as the Stiefel manifold.  Suppose that extremum is achieved by $L$, which has orthonormal basis $\{\bv_1,\cdots,\bv_k\}$.  The Lagrange multiplier equation says that $Q\bv_j\in L$ for any $j\in\{1,\ldots,k\}$.  That is to say, $L$ is invariant under $Q$.  This lemma follows from the standard property of symmetric matrices.
\end{proof}

\begin{defn}
On a Riemannian manifold, a smooth function $f$ is said to be $k$-convex at a point $p$ if the sum of the smallest $k$ eigenvalues of $\Hess(f)|_p$ is positive.
\end{defn}

It turns out that there is a naturally defined (semi-) two-convex function on the Atiyah--Hitchin manifold.

\begin{thm} \label{thm_unique}
In the Atiyah--Hitchin manifold $M$, the surface $\Sm$ is the only compact minimal 2-surface.  Also, there exists no compact, three-dimensional, minimal submanifold.
\end{thm}

\begin{proof}
Consider the square of the distance function to $\Sm$ with respect to \eqref{metric1}.  By \eqref{conn1},
\begin{align*}
\dd r^2 &= -2r\,\om^0 ~, \\
\Rightarrow\quad \Hess(r^2) &= 2\left( \om^0\ot\om^0 + r\frac{a'}{a}\,\om^1\ot\om^1 + r\frac{b'}{b}\,\om^2\ot\om^2 + r\frac{c'}{c}\,\om^3\ot\om^3 \right) ~.
\end{align*}
 Lemma \ref{estimate} and Lemma \ref{linear} imply that $r^2$ is two-convex when $r>0$.

Another way to derive the two-convexity of $r^2$, albeit only in a tubular neighborhood of $\Sigma$, is to apply \cite[Proposition 4.1]{ref_TsaiW2}, according to which strong stability of $\Sm$ implies that there exist positive constants $\vep$ and $\dt$ such that
\begin{align*}
\tr_L\Hess(r^2) &\geq \dt\,r^2
\end{align*}
at any point $p$ with $r\in[0,\vep)$, and any two-plane $L\subset T_pM$.  This can also be proved directly by using the expansions \eqref{series1}, and switching back to the rectangular coordinate for the fibers.

The rest of the argument is almost the same as that for \cite[Lemma 5.1]{ref_TsaiW}.  Suppose that $N\subset M$ is a compact minimal submanifold with dimension no less than $2$.  It follows from the semi-two-convextiy of $r^2$ that
\begin{align*}
\Delta^N(r^2|_N) = \tr_N( \Hess(r^2) ) \geq 0 ~.
\end{align*}
Appealing to the maximum principle, $r^2$ must be a constant on $N$.  Then, $\tr_N\Hess(r^2)$ vanishes.  This occurs only when $r^2$ vanishes on $N$.
\end{proof}

In view of the recent work of \cite{ref_LS}, the uniqueness theorem extends to the weaker setting of stationary integral varifolds. 

Here are some further remarks:
\begin{enumerate}
\item For the examples studied in \cite{ref_TsaiW}, the minimal submanifolds are totally geodesic and the corresponding $r^2$ is (semi-one-) convex.  It leads to a stronger rigidity phenomenon which does not hold true in the Atiyah--Hitchin manifold.

\item For small $r$, the series expansion of $\Hess(r^2)$ is derived for a general minimal submanifold in \cite[Proposition 4.1]{ref_TsaiW2}.  The second fundamental form appears as the coefficients of the linear term.  Unless it is a totally geodesic, $\Hess(r^2)$ cannot be semi-positive definite for small $r$.

\item Bates and Montgomery \cite{ref_BM} proved that the Atiyah--Hitchin manifold admits closed geodesics, and thus cannot support any convex function.

\item It can be shown that those examples of closed minimal $2$-spheres in hyper-K\"ahler K3 surfaces constructed by Foscolo \cite[Theorem 7.4]{ref_Fo} are indeed strongly stable.  The distance function to such a minimal $2$-surface is locally two-convex, and thus a local uniqueness theorem can be proved for these examples.

To say more, Foscolo proved that the minimal sphere still obeys \eqref{qtn}.  To validate Proof 2 of Proposition \ref{prop_sstable}, it remains to check that the minimal sphere has positive Gaussian curvature.  When the gluing parameter in \cite{ref_Fo} is sufficiently small, one can argue by continuity that the Gaussian curvature is still positive.

\item Dancer \cite{ref_Dancer} constructed non-trivial deformations of the hyper-K\"ahler metric on $M$.  Recently, G.~Chen and X.~Chen \cite{ref_CC} proved that Atiyah--Hitchin manifold and the Dancer's deformations are all the ALF-$D_1$ manifolds.  When the deformation parameter is small, it can be shown that the minimal $2$-sphere persists, and is still strongly stable and locally unique.  It is interesting to investigate the global uniqueness of the minimal $2$-sphere in Dancer's deformation.

\item The ALF-$D_0$ manifold is the quotient of $M$ by an isometric $\BZ/2$-action.  The image of $\Sm$ under the quotient map is a minimal $\BR\mathbb{P}^2$.  Since the $\BZ/2$ action is isometric, the corresponding statements of Proposition \ref{prop_sstable} and Theorem \ref{thm_unique} still hold true.  Namely, the minimal $\BR\mathbb{P}^2$ is strongly stable, and is globally unique.

\end{enumerate}

\end{document}